\documentclass[a4paper]{elsarticle}
\usepackage[utf8]{inputenc}%
\usepackage{amsmath}
\usepackage{amsfonts}
\usepackage{amssymb}
\usepackage[all]{xy}
\usepackage{mathbbol}
\usepackage{graphicx}
\usepackage{hyperref}%
\usepackage{amsthm}
\newtheorem{theorem}{Theorem}[section]

\newtheorem{corollary}[theorem]{Corollary}

\newtheorem{definition}[theorem]{Definition}

\newtheorem{lemma}[theorem]{Lemma}

\newtheorem{proposition}[theorem]{Proposition}

\usepackage{lipsum}
\makeatletter
\def\ps@pprintTitle{%
 \let\@oddhead\@empty
 \let\@evenhead\@empty
 \def\@oddfoot{}%
 \let\@evenfoot\@oddfoot}

\begin{document}
\title{Concentration Inequalities in Riesz Spaces}
\tnotetext[t1]{This document is the results of the research
project funded by the National Science Foundation.}
\tnotetext[t2]{The authors are members of the GOSAEF research group.}
\author[1]{Mohamed Amine BEN AMOR\corref{cor1}%
\fnref{fn1}}
\ead{mohamedamine.benamor@ipest.rnu.tn}

\author[2]{Amal OMRANI\fnref{fn2}}
\ead{amal.omrani@ipeiem.rnu.tn}

\fntext[fn1]{The author is member of the GOSAEF research group. \url{www.gosaef.com}}
\fntext[fn2]{The author is member of the GOSAEF research group. \url{www.gosaef.com}}

\address[1]{Research Laboratory of Algebra, Topology, Arithmetic, and Order,  and
GOSAEF, Department of Mathematics, Faculty of Mathematical,
Physical and Natural Sciences of Tunis, Tunis-El Manar University,
2092-El Manar, Tunisia.}
\address[2]{Research Laboratory of Algebra, Topology, Arithmetic, and Order,  and
GOSAEF, Department of Mathematics, Faculty of Mathematical,
Physical and Natural Sciences of Tunis, Tunis-El Manar University,
2092-El Manar, Tunisia.}

\begin{abstract}
In this work, we will generalize the moment generating function to Riesz spaces. We will derive some of its properties and use it to prove concentration inequalities on Riesz spaces.
\end{abstract}


\maketitle

\section{Introduction}

Various topics in stochastic processes have been considered in the abstract setting of Riesz spaces.
Labuschagne and Watson in \cite{kuo2005conditional} define conditional expectation operators as positive order-continuous projections mapping weak order units to weak order units and having Dedekind complete range. With this definition, conditional expectation operators are shown to commute with certain band projections. The averaging properties of these operators are then shown, which leads to the extension of the domains of such operators to what is called their maximal domain. This definition of conditional expectation was used to generalise martingales, submartingales, stopping times and optional stopping theorems to vector lattices. In \cite{kuo2005zero} the concept of independence was generalised, as well as the Borel-Cantelli Lemma and Kolmogorov’s Zero-One Law. 
Kuo, Vardy and Watson generalised Markov processes \cite{vardy2012markov} and Bernoulli processes, with a related law of large numbers, the Bienaymé inequality, and Poisson’s theorem \cite{kuo2016bernoulli} to Riesz spaces.

By contrast Concentration Inequalities which lead to statistical applications have received very little attention. In this work, we prove some of the concentration inequalities in Riesz spaces: Chernoff inequality, Bennett's inequality and Hoeffding inequality. We define among other the moment generating function for bounded elements.
%
%
%

The next section will be devoted to some preliminaries on Riesz spaces, representation theorem and representation theorem on Riesz spaces. Later, we will construct the exponential function on Riesz space and derive some of its properties. The exponential function will play a key role on our studies. We will generalize the well known Moment generating function to the framework of measure free Riesz spaces and prove on it its most relevant properties. The three final sections, will be devoted to the concentration inequalities on Riesz spaces.

\section{Preliminaries}
This section is devoted to some preliminaries on representation theorem in Riesz Spaces and Conditional expectation in Dedekind complete Riesz Spaces. However,  the reader is expected to be familiar with the basic theory of Riesz Spaces. We refer to the classical monographs \cite{aliprantis2006positive}, \cite{luxemburg1971ac} and \cite{zaanen2012introduction} for undefined terminology.

From now, we will assume that $E$ is a Dedekind complete Riesz space with $u$ as a weak order unit. $E_u$ will denote the order ideal generated by $u$. (and then $u$ becomes a strong unit in $E_u$).
We recall that $E_u$ can be equipped by an $f$-algebra multiplication in such a manner that $u$ becomes an algebra unit. Yoshida proved in his early work \cite{yosida194270} the following representation theorem which will be useful in the sequel. (For more details about the representation theorem see \cite{groenewegen2016spaces}, \cite{luxemburg1971ac} and \cite{yosida194270}  ).

\begin{theorem}\label{yos}
Let $E_u$ be a Dedekind complete Riesz Space with a strong unit $u$. Then there are a compact space $X$ and a Riesz isomorphism $\varphi:\ E_u \to C(X)
$ such that\[\begin{array}{lcl}
\varphi(u)&=&\mathbb{1}
\end{array}
\]
\end{theorem}

As a corollary of the Yoshida theorem we can endow $E_u$ with an $f$-algebra multiplication such that $u$ becomes an  algebra unit and  $\varphi$  an algebra homomorphism. See \cite{groenewegen2016spaces} and \cite{luxemburg1971ac} for more details.

$\mathcal{C}(X)$ equipped with $\|.\|_\infty$ is a Banach spaces. If we equip $E_u$ with the Jauge norm : $\|.\|_u$ defined as: \[\|f\|_u=\inf{\lbrace\beta \in \mathbb{R},\text{ such that }|f|\leq \beta u\rbrace},\] then $\varphi$ is also an isometry.

Let now recall some facts about the conditional expectation and independence in Riesz spaces. For more details we can refer to \cite{kuo2004discrete} and \cite{kuo2005conditional}.

Let $E$ be Dedekind complete Riesz space with $u$ as weak order unit. We call $P$ and $Q$, \textit{$T$-independent} band projections in $E$ whenever \[TPQu=TPu\ TQu\] holds.

We say that two Riesz subspaces $E_1$ and $E_2$ of $E$ are $T$-conditionally independent if all band projections $P_i,$  such that $P_i(u) \in E_i$ for $i=1,2$ are $T$-conditionally independent .

It should be noted that $T$-conditional independence of the band Projection $P$ and $Q$ is equivalent to $T$-conditional independence of the closed Riesz subspace $<Pu,R(T)>$ and $<Qu,R(T)>$  generated by $Pu$ and $R(T)$ and by $Qu$ and $ R(T)$ respectively.

The concept of $T$-conditional independence can be extended to a family $(E_\lambda)_{\lambda \in \Lambda}$ of closed Dedekind complete Riesz spaces of $E$ with $R(T)\subset E_\lambda$ for all $\lambda \in \Lambda$. We say that the family is $T$-conditionally independent, if for each pair of disjoint subsets $\Lambda_1$ and $\Lambda_2$ of $\Lambda$, we have that  $E_{\Lambda_1}$ and $E_{\Lambda_2}$ are $T$-conditional independent, where $E_{\Lambda_j} := <\cup_{\lambda \in \Lambda_j} E_\lambda>$ for $j=1,2$.

Finally, we say that a sequence $(f_n)$ in $E$ is $T$-conditionally independent if the family of closed Riesz spaces $<{f_n}\cup R(T)>,\ n\in \mathbb{N}$, is $T$-conditionally independent.

For the convenience of the reader we repeat the next lemma and definition from \cite{kuo2016bernoulli} without proofs, thus making our exposition self-contained.

\begin{lemma}\label{indop}
Let $E$ be a $T$-universally complete Riesz space with weak order unit $u=Tu$ where $T$ is a strictly positive conditional expectation operator on $E$. Let $f$ and $g \in E_u$. If $f$ and $g$ are $T$-conditionally complete independent then \[Tfg = TfTg = TgTf\] holds.
\end{lemma}

\begin{definition}\label{bernou}
Let $E$ be a Dedekind Riesz space with weak order unit, $u$, and conditional expectation operator $T$ with $Tu = u$. Let $(P_k )_{k\in \mathbb{N}}$ be a sequence of $T$-conditionally independent band projections. We say that$(P_k )_{k\in \mathbb{N}}$ is a Bernoulli process if
\[TP_u = f \]for all $k\in \mathbb{N}$ for some fixed $f\in E_u$.
\end{definition}

\section{Exponential function in Riesz Spaces}

\begin{center}
Throughout this paper we denote  by $x^0$ the unit element $u$ in $E_u$.
\end{center}

There are several methods to define the exponential function on a Dedekind complete riesz space. We can cite among others the function calculus method (see \cite{buskes1991functional} or \cite{grobler1988functional}). We choose to use the Yoshida theorem \ref{yos}. 

\begin{theorem}\label{expo}
Let $E_u$ be a Dedekind complete Riesz Space with a strong unit $u$. Then the power serie \[S_n(x)=\sum_{k=0}^n \frac{1}{k!}x^k\] o-converges for every $x$ in $E_u$. Its limit will be denoted $\exp(x)$.
\end{theorem}

\begin{proof}
According to The Yoshida representation Theorem \ref{yos}, there are a compact space $X$ and an algebra and a Riesz isomorphism $\varphi: \ E_u \to C(X)$ such that $\varphi(u)=1$. It follows that for every natural number $n\geq 0$, we have: 
\[
\begin{aligned}
\varphi(S_n(x))&=\sum_{k=0}^n \frac{1}{k!}\varphi(x^k)\\
 &= \sum_{k=0}^n \frac{1}{k!}\varphi(x)^k\\
 &= \sum_{k=0}^n \frac{1}{k!}\hat{x}^k\\
 &= S_n(\hat{x})
\end{aligned}
\]
where $\hat{x}=\varphi(x)$.

Since $-\lambda u\leq x\leq\lambda u$, it follows that $\|\hat{x}\| \leq \lambda$. Consequently, $(S_n(\hat{x}))_n$ is uniformly convergent to $exp(\hat{x})$, i.e.  for all $\varepsilon > 0$ , there exists $N_0\in\mathbb{N}$ such that for all $n\geq N_0$ : 
\[ \|\sum_{k=0}^n \frac{1}{k!}\hat{x}^k - exp(\hat{x})\| \leq \varepsilon. \]
and so for all $\displaystyle t\in X : \left|\sum_{k=0}^n \frac{1}{k!}\hat{x}^k(t) - \exp(\hat{x})(t)\right| \leq \varepsilon \mathbb{1}(t) $.

By Theorem 1 in \cite{yosida194270} we get \[\left|\sum_{k=0}^n \frac{1}{k!} x^k - \exp(x)\right| \leq \varepsilon u.\]
so that $(S_n(x))_n$ is  ru-convergent to $\exp(x)$ and hence o-convergent.
\end{proof}

The map $\exp:\ E_u \to E_u$ that maps every element $x$ in $E_u$ to $\exp(x)$ is well defined and one to one, since $\varphi$ and $\exp$ are one to one.
$$
\xymatrix{
  E_u \ar[r]^{\varphi}   & \mathcal{C}(X) \ar[d]^{\exp} \\
   & \mathcal{C}(X) \ar[lu]^{\varphi^{-1}}
   }
$$
The $\exp$ function appears to be in $E_u$ as \[\exp=\varphi^{-1} \circ \exp\circ \varphi\]
  In the next proposition we give a series of properties that the $\exp$ function verifies in $E_u$.
\begin{proposition}\label{propexp}
Let $E_u$ be a Dedekind complete Riesz Space with a strong unit $u$. The following statements hold:
\begin{enumerate}
\item for all $x$ and $y$ in $E_u$, $\exp(x+y)= \exp(x)\exp(y)$.
\item for all $x$ in $E_u$, $exp(x)\geq 0$.
\item for all $x$ in $E_u$, $exp(x)$ is invertible.
\end{enumerate}
\end{proposition}

\begin{proof}
\begin{enumerate}
\item Let $x$ and $y$ in $E_u$. Since $\varphi$ and $\varphi^{-1}$ are Riesz and ring homomorphisms we get :\[\begin{array}{lcl}
\exp(x+y)&=&\varphi^{-1} \circ \exp\circ \varphi(x+y)\\
 &=& \varphi^{-1} (\exp(\varphi(x))\exp(\varphi(y)))\\
 &=&  \varphi^{-1} (\exp(\varphi(x))) \varphi^{-1} (\exp(\varphi(y))) \\
 &=& \exp(x)\exp(y).
\end{array}\]
Which is the desired result.
\item For every $x$ in $E_u$, we have $\exp(x)=\varphi^{-1} \circ \exp\circ \varphi (x)$. The result follows from the positiveness of the function $\exp$ in $\mathcal{C}(X)$.
\item from the first point, we get that \[u=\exp(x-x)=\exp(x)\exp(-x)\] for every $x$ in $E_u$. This yields to the fact that $\exp(x)$ is invertible in $E_u$ and its inverse is $\exp(-x)$.
\end{enumerate}
\end{proof}

The next technical proposition will play a key role in the next section.

\begin{proposition}\label{band}
Let $E_u$ be a Dedekind complete Riesz Space with a strong unit $u$, Then for all $x$ and $y$ in $E_u$ , there exists a positive invertible element $z$ in $E_u$ such that  $exp(x) - exp(y) = z(x-y) .$ 
\end{proposition}

\begin{proof}
Let $x$ and $y$ be in $E_u$. We will denote $\hat{x}$ and $\hat{y}$ their representant is $\mathcal{C}(X)$ respectively. Let \[\hat{z}(t)=\left \{\begin{array}{lcl}
\frac{\exp(\hat{x}(t))-\exp(\hat{y}(t))}{\hat{x}(t)-\hat{y}(t)} &\text{ if } &\hat{x}(t) \neq \hat{y}(t)\\
& & \\
\exp(\hat{x}(t))  &\text{ if }& \hat{x}(t)= \hat{y}(t)
\end{array}\right. \] for every $t$ in $X$. Observe that $\displaystyle \hat{z}(t)=\int_0^1 \exp(s\hat{x}(t)+(1-s)\hat{y}(t))ds$ which is continuous and strictly positive from the classical Lebesgue theorems. It follows that \begin{equation}\exp(\hat{x})-\exp(\hat{y})=\hat{z}(\hat{x}-\hat{y})
\label{eq11}\end{equation}  and $\hat{z}$ and $\frac{1}{\hat{z}}$ are both in $\mathcal{C}(X)$.

Composing (\ref{eq11}) by $\varphi^{-1}$, we get the desired result.
\end{proof}

\section{Moment generating function in Riesz Spaces}

Once we defined the exponential function on a Dedekind complete Riesz space with strong order unit $u$, $E_u$, we are able to define the Moment generating function on it.

\begin{definition}[Moment generating function]
Let $T$ be a conditional expectation on $E$. For every $x\in E_u$, we define the map $\mathcal{M}_x: \mathbb{R} \to E_u$, by $$\mathcal{M}_x: t\mapsto T(exp(tx))$$
$\mathcal{M}_x$ will be called \textit{the moment generating function} of $x$.
\end{definition}

We are widely inspired from the Lemma 4.1 in \cite{kuo2016bernoulli} to prove the next Lemma.

\begin{lemma}\label{puissance}
Let $T$ be a strictly positive conditional expectation on a Dedekind complete Riesz space with strong order unit $u$, $E_u$. If $f$ and $g$ are two independent elements then \[Tf^ng^m=Tf^nTg^m=Tg^mTf^n\] holds for every natural numbers $n$ and $m$.
\end{lemma}

\begin{proof}
Since $f$ and $g$ are $T$-independent, it follows that the closed Riesz subspaces $E_f=<\mathcal{R}(T),f>$ and  $E_g=<\mathcal{R}(T),g>$ generated by $ \mathcal{R}(T)$ and $f$ and by $\mathcal{R}(T)$ and $g$ respectively are $T$-independent. From the Radon-Nikodym Theorem (see \cite{watson2009ando}), there exist two conditional expectation $T_f$ and $T_g$ with ranges $E_f$ and $E_g$ respectively such that \[ T= T_fT_g=T_gT_f\]

We will observe first that $Tf^n =T_gf^n$. Indeed \[\begin{array}{lcl}
Tf^n&=&T_gT_f f^n\\
 & = &T_g f T_f f^{n-1}\\
 & = &\dots\\
 & = & T_g f^n
\end{array} \] as $f$ is in $E_f$. As $f$ and $g$ play symmetric roles we can affirm that $Tg^m =T_fg^m$.

Now, we can use the latter fact to prove the desired result. Indeed,\[\begin{array}{lcl}
Tf^ng^m &= & T_gT_f f^ng^m \\
& = &  T_g f T_f f^{n-1} g^m \text{ as $f$ is in $\mathcal{R}(T_f)$}\\
& = & T_g f^n T_f g^m\\
& = & T_g f^n Tg^m \\
&= & T g^m T_g f^n  \text{ as $\mathcal{R}(T)\subset \mathcal{R}(T_f)$}\\
 &= & Tg^mTf^n
\end{array} \]which makes an end to our proof.

\end{proof}
At this point, we are able to prove the main result of this section.
\begin{theorem}\label{ext}
If $f$ and $g$ are two $T$-independent elements in the Dedekind complete Riesz space with strong unit $u$, then \[\mathcal{M}_{f+g}=\mathcal{M}_f\ \mathcal{M}_g\]
holds.
\end{theorem}

\begin{proof}
Let $f$ and $g$ be two $T$-independent elements in $E_u$, and $t$ a real number then \[
\begin{aligned}
\mathcal{M}_{f+g}(t) &=& T\exp(tf+tg)\\
 & =& T\exp(tf)\exp(tg)\\
 & =& T \sum_{k=0}^\infty \frac{(tf)^k}{k!} \sum_{j=0}^\infty \frac{(tg)^j}{j!}\\
\end{aligned}\]
From the order continuity of $T$ it follows that \[
\mathcal{M}_{f+g}(t) = \sum_{k=0}^\infty \sum_{j=0}^\infty \frac{t^k}{k!} \frac{t^j}{j!} T(f^k g^j)\]
Lemma \ref{puissance} and the order continuity of $T$ again yield to \[\begin{aligned}
\mathcal{M}_{f+g}(t) &= &\sum_{k=0}^\infty \sum_{j=0}^\infty \frac{t^k}{k!} \frac{t^j}{j!} Tf^k Tg^j \\
& = & T\sum_{k=0}^\infty \frac{(tf)^k}{k!} T\sum_{j=0}^\infty \frac{(tg)^j}{j!} \\
&=& \mathcal{M}_f(t) \mathcal{M}_g(t)
\end{aligned}\]
And we are done.
\end{proof}

\section{The Chernoff inequality in Riesz Spaces}
We start our study with the following technical lemma:
\begin{lemma}\label{exband}
Let $E_u$ be a Dedekind complete Riesz space with a strong unit u, then for all $x$ and $y$ in $E_u$, the projection band generated by $(x-y)^+$ is equal to the projection band generated by $(\exp(\lambda x)-\exp(\lambda y))^+$.
\end{lemma}

\begin{proof}
Proposition \ref{band} yileds to \[\exp(x)-\exp(y) = z(x-y)\] for some invertible positive element $z$. It follows that \[(\exp(x)-\exp(y))^+ = z(x-y)^+\]and then,

\[\{ (\exp(x)-\exp(y))^+\}^{\perp\perp}= \{ (x-y)^+\}^{\perp\perp}\]
which makes an end to our proof.
\end{proof}
\begin{lemma}\label{indep}
Let $E_u$ be a Dedekind complete Riesz space with a strong unit $u$ and $T$ a strictly positive conditional expectation on $E_u$. Let $(P_j)_{j\in \mathbb{N}}$ be a $T$-conditionally independent band projections with $TP_j u=f$ for all $j\in \mathbb{N}$ for some fixed $f$ in $E_u$. Then for any stricly positive real number $\lambda$ and all $n \in \mathbb{N}$, we have the following equality:
\[ T\prod_{i=1}^n \exp(\lambda P_i u) = \prod_{i=1}^n(u+(\exp(\lambda)-1)f)\] 
holds.
\end{lemma}
\begin{proof}
Using the definition of the exponential in Theorem \ref{expo}, we have :
\[\begin{aligned}
\exp(\lambda P_i u) &= \sum_{k=0}^\infty \frac{(\lambda P_i u)^k}{k!}\\
&=\sum_{k=0}^\infty \frac{(\lambda )^k}{k!} P_i u +u-P_i u \\
&= u + P_i u (\exp(\lambda) - 1).
\end{aligned}\]
 As a result, we get :
\[ \prod_{i=1}^n \exp(\lambda P_i u) =  \prod_{i=1}^n (u+\alpha P_i u ) = \sum_{k=0}^n \alpha \sigma_k\]
 where \[\alpha = \exp(\lambda)-1\] and \[\sigma_k = \sum_{1\leq i_1<..<i_k\leq n} \alpha P_{i_1} u...P_{i_k} u.\] 
 Let $\sigma_0=u$

Finally the $T$-conditional independence of $P_1,...,P_n$ and lemma (\ref{indop}) applied iteratively give 
\[ T(\prod_{i=1}^n (u+\alpha P_i u )) = T(\sum_{k=0}^n \alpha^k \sigma_k ) = \sum_{k=0}^n \alpha^k T(\sigma_k) = \prod_{i=1}^n (u+\alpha f)\]which make an end to our proof.
\end{proof}

At this point, we gathered all the ingredients we need to prove the main result of our work.
\begin{theorem}[Cherrnoff's inequality]
Let $E_u$ be a Dedekind complete Riesz space with a strong unit $u$ and $T$ a strictly positive conditional expectation on $E_u$. Let $(P_j)_{j\in \mathbb{N}}$ be a Bernoulli process with $TP_ju=f$ for all $j\in \mathbb{N}$ for some fixed $f$ in $E_u$ and let $S_n=\sum_{j=0}^n\ P_j u$ then 
\[ TP_{(S_n-tu)^+}u\leq\left(\frac{ne\|f\|_u}{t}\right)^t \exp(-nf)\] holds for any strictly positive scalar $t$ such that $t > n\|f\|_u$.
\end{theorem}
\begin{proof}
From Lemma \ref{exband}, one can deduce the following equality :
\[\begin{array}{lcl}
TP_{(S_n-tu)^+} u &=& TP_{(\exp(\lambda S_n)-\exp(\lambda u))^+} u.
\end{array}\]
Notice  that $P_{(\exp(\lambda S_n)-\exp(\lambda t)u)^+}$ is the band projection onto the band generated by $(\exp(\lambda S_n)-\exp(\lambda t)u)^+$. It follows that \[P_{(\exp(\lambda S_n)-\exp(\lambda t)u)^+}(\exp(\lambda S_n)-\exp(\lambda t)u)\geq 0\] Then, \begin{equation}
 P_{(\exp(\lambda S_n)-\exp(\lambda t)u)^+}(\exp(\lambda S_n))\geq \exp(\lambda t)P_{(\exp(\lambda S_n)-\exp(\lambda t)u)^+} u \geq 0    
\label{eq1}
\end{equation}
Since Band projections are dominated by the identity map, it follows that  \begin{equation}\exp(\lambda S_n)\geq P_{(\exp(\lambda S_n)-\exp(\lambda t)u)^+}(\exp(\lambda S_n))   
\label{eq2}
\end{equation}
Combining (\ref{eq1}) with (\ref{eq2})  and applying $T$ we obtain :
\begin{equation}
TP_{(S_n-tu)^+} u \leq \exp(-\lambda t) T(\exp(\lambda S_n)) .
\label{eq3}
\end{equation}
Now using the first property of the exponential \ref{propexp} to move from a sum into a product  then using the independence of the Bernoulli process and  \ref{indep} we get :
$$\begin{aligned}
\exp(-\lambda t) T(\exp(\lambda S_n)) &=\exp(-\lambda t)T(\prod_{i=1}^n(\exp(\lambda P_i u))\\ 
&= \exp(-\lambda t)\prod_{i=1}^n(u+(\exp(\lambda)-1)f).
\end{aligned}$$
However since  $1+\hat{x}\leq \exp(\hat{x} )$ holds, for all $\hat{x} \in \mathcal{C}(X)$ it follows that   $1+x\leq \exp(x)$ holds for all $x\in E_u$. As a result we deduce that :
\[\exp(-\lambda t)\prod_{i=1}^n(u+(\exp(\lambda)-1)f)\leq \exp(-\lambda t)\exp(nf (\exp(\lambda)-1))\] 
Substuting this into (\ref{eq3}) we obtain
\[TP_{(S_n-tu)^+} u \leq \exp(-\lambda t)\exp(nf (\exp(\lambda)-1))\] 
This bound holds for any $\lambda > 0$, particularly for $\displaystyle \lambda = -\log\left(\frac{n\|f\|_u}{t}\right).$ This yields to \[TP_{(S_n-tu)^+} u \leq \left(\frac{n\|f\|_u}{t}\right)^t\exp\left(\frac{tf}{\|f\|_u}-nf\right).\]
Since $\displaystyle \frac{f}{\|f\|_u} \leq u$, it follows that  \[TP_{(S_n-tu)^+} u \leq \left(\frac{n\|f\|_u}{t}\right)^t\exp\left(tu-nf\right),\] which makes an end to our proof.
\end{proof}
\section{Bennett's inequality in Riesz Spaces}
In this section, we present another concentration inequality in Riesz space with unit: The Bennett's inequality. In this order we need to define the logarithm function in Riesz space.

Notice first that if $f$ is a positive invertible element in the Dedekind complete Riesz Space $E_u$ with unit $u$, then $\hat{f}$, its representant in $\mathcal{C}(X)$, is strictly positive. The following definition follows.

\begin{definition}
Let $E_u$ be a Dedekind complete Riesz space with a strong order unit $u$ . Define the logarithm function on Riesz space as follows:
For every positive invertible element $f$ in $E_u$ :\[\log(f)=\varphi^{-1} \circ \log \circ \varphi(f)\]
\end{definition}

The next proposition present some of the properties of the logarithm function on $E_u$. We leave the proof for the reader.
\begin{proposition}\label{logo}
Let $E_u$ be a Dedekind complete Riesz space with a strong order unit $u$ . The following statements holds:
\begin{enumerate}
\item For every positive invertible two elements $x$ and $y$ in $E_u$ \[\log(xy)=\log(x)+\log(y)\]
\item The inverse function of the exponential function on Riesz space is the logarithm function .
\item $x\longmapsto \log(u+x)$ is a concave function for all $x>-u$. 
\end{enumerate}
\end{proposition}

The next technical lemma will play a key role in the proof of the main result of this section.

\begin{lemma}
Let $E_u$ be a Dedekind complete Riesz Space with unit $u$ and $T$ be a conditional expectation, then $T(\exp(f))$ is invertible for any $f$ in $E_u$.
\end{lemma}

\begin{proof}
We will proceed once again by the way of representation. Pick $f$ in $E_u$ and let $\hat{f}$ be its representant in $\mathcal{C}(X)$ (see \ref{yos}). Since $X$ is compact, it follows that there is some $\alpha$ in $\mathbb{R}$, such that $\hat{f}\geq \alpha \mathbb{1}$. If we apply the exponential function on the last result, we get $\exp(\hat{f})\geq \exp(\alpha) \mathbb{1}$. Again with \ref{yos}, we obtain that \[\exp(f)\geq \exp(\alpha) u. \]
It follows that \[T\exp(f)\geq \exp(\alpha) u. \]
Lemma 5.9 in \cite{huijsmans1991lattice} yields to the desired result.
\end{proof}

\begin{lemma}\label{phi}
Let $E_u$ be a Dedekind complete Riesz space with a strong order unit $u$. Let $\Phi$ the map: \[\begin{array}{llcl}
\Phi : & E_u & \longrightarrow & E_u\\
 & f & \longmapsto & \exp(f) - f - u 
\end{array}\] For every element $f$ in $E_u$ such that $f\leq u$ the following inequality \[\Phi(tf)\leq f^2\Phi(tu)\]
 holds for every strictly positive real number $t$.
\end{lemma}
\begin{proof}
We will proceed again by the way of representation. Notice that $\varphi\left( \Phi(tf)\right)=\Phi(t\hat{f})$. A straightforward calculus yields to $\Phi(t\hat{f}) \leq \hat{f}^2 \Phi(t)$. The result follows by applying $\varphi^{-1}$.
\end{proof}

At this point we are able to prove the main result of this section
\begin{theorem}[Bennett's inequaliy]
Let $E_u$ be a Dedekind complete Riesz space with a strong unit $u$ and $T$ a strictly positive conditional expectation on $E_u$. Let $f_1,\dots ,f_n$ be $n$ independent elements in $E_u$.
Let \[S=\sum_{k=1}^n (f_i - T(f_i))\]
and \[ v =\sum_{k=1}^n T(f_i ^2)\]

Then, for any real $t>0$, we get \[\Psi_S(t) := \log(T(\exp(tS)))\leq v \Phi(t)\]
 
 and if $v$ is invertible then for all positive element $x$ in $E_u$, we have \[TP_{(S-xu)^+} u \leq  \exp[-\|v\|_u[(1+\frac{x}{\|v\|_u})\log(1+\frac{x}{\|v\|_u})-\frac{x}{\|v\|_u}]]\]
\end{theorem}
\begin{proof}
In order to prove Bennett's inequality, we begin by proving that $\psi_S(t)$ is bounded.We have : 
$$\begin{aligned}
\Psi_S(t)&=\log(T\exp(t\sum_{k=1}^n (f_i - T(f_i))  ))\\
&=\log(T(\prod_{k=1}^n \exp(t(f_i - T(f_i)) ) ).
\end{aligned}$$ The last inequality is a consequence of the first property from proposition (\ref{propexp}).

Note that $f_1,..,f_n$ are independent and   $<f_i,R(T)> = <f_i-T(f_i),R(T)>$ for all $i\in\lbrace 1,..,n \rbrace$, so that $f_1-T(f_1),..,f_n-T(f_n)$ are independent. Thus from theorem (\ref{ext}) we obtain:
 $$\begin{aligned}
\Psi_S(t)&=\log[\prod_{k=1}^n T(\exp(t(f_i - T(f_i)) ) ]\\
&=\log[\prod_{k=1}^n T(\exp(t(f_i))(\exp( - tT(f_i)) ) ].
\end{aligned}$$
We next claim, as a consequence of the first property from lemma (\ref{logo}), that:
 \[\log[\prod_{k=1}^n T(\exp(t(f_i))(\exp( - tT(f_i)) ) ]=\sum_{k=1}^n\log[ T(\exp(t(f_i))\exp( - tT(f_i)) ) ].\]
Recall that the range of T is order closed so that if an element belongs to $R(T)$ then its exponential belongs to it as well. Hence the averaging property leads to: \[T[\exp(t(f_i))\exp( - tT(f_i))]=\exp( - tT(f_i))T(\exp(t(f_i)).\]  Moreover, the second property of lemma  (\ref{logo}) implies:  \[\log[\exp( - tT(f_i))]= -tT(f_i),\forall i\in {1,..,n}.\]  Consequently: 
$$\begin{aligned}
\Psi_S(t)&=\sum_{k=1}^n\log[T(\exp(t(f_i))\exp( - tT(f_i)) ]\\
&=\sum_{k=1}^n\log[T(\exp(t(f_i))\exp( - tT(f_i)) ]\\
&=\sum_{k=1}^n\log[T(\exp(t(f_i))]- tT(f_i) .
\end{aligned}$$
As $\exp(tf_i)\leq u + tf_i + (e^t -t -1)f_i^2 ,\forall i\in{1,..,n}$, T is strictly positive and the logarithm is an increasing  function,we have:
\[\log(T(exp(tf_i)))\leq\log[u+tT(f_i) +(e^t-t-1)T(f_i^2)]\]
Hence that:
\[\Psi_S(t)\leq\sum_{k=1}^n\log[u+tT(f_i) +(e^t-t-1)T(f_i^2)]-tT(f_i)\]
Finally as  the function $x\longrightarrow \log(u+x)$ is concave for all $x>-u$ we conclude that :
$$\begin{aligned}
\Psi_S(t)&\leq n(\log[u+\frac{t}{n}\sum_{i=1}^n T(f_i) + (e^t-t-1)\frac{v}{n}]) - t\sum_{i=1}^n T(f_i)\\
&\leq(e^t-t-1)v. 
\end{aligned}$$
We can now proceed analogously to the proof of Chernoff inequality, so that   for all $x >0$ :
$$\begin{aligned}
TP_{(S-xu)^+} u &\leq \exp(-tx) T(\exp(tS))\\
&= \exp(-tx)\exp(\psi_S(t))\\
&=\exp(-tx+\psi_S(t))\\
&\leq\exp(-tx+v(e^t-t-1)).
\end{aligned}$$
The proof is completed by showing that: \[\exp(-tx+v(e^t-t-1))\leq \exp[-\|v\|_u[(1+\frac{x}{\|v\|_u})\log(1+\frac{x}{\|v\|_u})-\frac{x}{\|v\|_u}]]\]
We see that the inequality $\varphi(TP_{(S-ux)^+} u) \leq \exp(-tx) \exp(\hat{v(s)} (e^t -t -1))$  holds for all $s$ in $C(X)$ which is clear as $\varphi$  is monotonous. In particular it holds for $\|\hat{v}\|_{\infty}$.\\
The right bound side is optimized for $t=\log(1+\frac{x}{\|\hat{v}\|_{\infty}})$ which is well defined, because $v$ is invertible therefore  $\|\hat{v}\|_{\infty} $ is non null.
It follows that : 
\[\varphi(TP_{(S-ux)^+} u) \leq \exp[-\|\hat{v}\|_{\infty}[(1+\frac{x}{\|\hat{v}\|_{\infty}})\log(1+\frac{x}{\|\hat{v}\|_{\infty}})-\frac{x}{\|\hat{v}\|_{\infty}}]]\]
Similarly, the monotony of  $\varphi^{-1}$  gives the desired bound and complete the proof.

\end{proof}
\section{Hoeffding's inequality in Riesz spaces}
\begin{definition}
Let $E_u$ be a Dedekind complete Riesz space with a strong unit $u$ and $T$ a strictly positive conditional expectation on $E_u$. An element $X$ in $E_u$ is called subGaussian with parameter $v$ where $v$ is an invertible element of $E_u$ , if for all $\lambda \in \mathbb{R}$ if \[\Psi_{X-T(X)} (\lambda)\leq \frac{\lambda^2}{2} v \]
\end{definition}
\begin{proposition}\label{subgau}
Let $E_u$ be a Dedekind complete Riesz space with a strong unit $u$ and $T$ a strictly positive conditional expectation on $E_u$. If X is subGaussian with parameter $v$ then  for all $\lambda \in\mathbb{R}$ we have: \[ TP_{((X-T(X))-tu)^+} u \leq \exp(-\frac{\lambda^2}{2\|v\|_u})u\]
\end{proposition}
\begin{proof}
It is a simple matter -using the Chernoff technique-  to show that for any strictly positive number $s$, we have  \[TP_{(X-T(X))-tu)^+} u \leq \exp(-st)\exp(\psi_{X-T(X)}(s))\]  
Next we use the fact that $X$ is subGaussian with parameter $v$ to get \[TP_{(X-T(X))-tu)^+} u \leq \exp(-st)\exp(\frac{s^2}{2}v)\]
We can now use Yoshida representation and the techniques  used in previous theorems to show that  \[\varphi ( TP_{((X-T(X))-tu)^+} u) \leq \exp(-st+\frac{s^2}{2}\|\hat{v}\|_{\infty} )\]
We wish to make the inequality the tightest possible , thus we minimize with respect to $s>0$ solving $\Phi'(s)=0$, where $\Phi(s)=-st+\frac{s^2}{2}\|\hat{v}\|_{\infty}.$
 
 We find that $\inf\Phi(s)=-\frac{t^2}{2\|v\|_{\infty}}$.

 This proves that  \[\varphi ( TP_{((X-T(X))-tu)^+} u) \leq \exp(-\frac{t^2}{2\|v\|_{\infty}})\]
Hence \[ TP_{((X-T(X))-tu)^+} u \leq \exp(-\frac{t^2}{2\|v\|_u}) u.\] wish is the desired inequality.
\end{proof}
\begin{theorem}
Let $E_u$ be a Dedekind complete Riesz space with a strong unit $u$ and $T$ a strictly positive conditional expectation on $E_u$ with $Tu=u$ . Let  $X_1,..,X_n$ be $n$ T-independent element of $E_u$ such that $X_i$ is subGaussian with parameter $v_i$ for all $i \in {1,..,n} $. Then, for any strictly positive scalar $t$ we have :
\[TP_{(\sum_{i=1}^n (X_i-T(X_i)) - tu)^+} u \leq \exp(-\frac{t^2}{2\sum_{i=1}^n\|v_i\|_u})u\]
\end{theorem}
\begin{proof}
Let us first use Chernoff's technique thus the following inequality holds for all $\lambda>0$ \[TP_{(\sum_{i=1}^n Y_i - tu)^+} u \leq \exp(-\lambda t)T(\exp(\lambda \sum_{i=1}^n Y_i))\] where $Y_i=X_i-T(X_i)$

Since $Y_1,..,Y_n$ are T-independent and by lemma (\ref{ext}) applied iteratively, we show that \[T(\exp(\lambda \sum_{i=1}^n Y_i)) = \prod_{i=1}^n T(\exp(\lambda Y_i))\]
But, for all $i\in{1,..,n}$  \[ \log(T(\exp(\lambda Y_i))) \leq \frac{\lambda^2}{2}v_i\] because $Y_i$ is subGausian with parameter $v_i$.

Consequently we get: \[ \sum_{i=1}^n\log(T(\exp(\lambda Y_i))) \leq \frac{\lambda^2}{2}\sum_{i=1}^n v_i\]  So that \[TP_{(\sum_{i=1}^n Y_i - tu)^+} u\leq\exp(-\lambda t)\exp(\frac{\lambda^2}{2}\sum_{i=1}^n v_i)\]
The proof is completed by minimizing the right part of the inequality proceeding the same way as in the proof of (\ref{subgau}). 
\end{proof}
\begin{lemma}
Let $E_u$ be a Dedekind complete Riesz space with a strong unit $u$ and $T$ a strictly positive conditional expectation on $E_u$. For $X$in $[au,bu]$ where $a$ and $b$ are two scalars. Then for any $\lambda$in $\mathbb{R}$ we have:
\[\Psi_{X-T(X)}(\lambda) \leq \frac{\lambda^2(b-a)^2}{8}u.\]
\end{lemma}
\begin{proof}
The main idea of the proof is to use Yoshida representation. First, note that $\varphi\circ T\circ\varphi^{-1}$ is a positive linear  continuous operator because $\varphi$ is an isometric function and $T$ is a positive linear order continuous operator.

let $\tilde{T} = \varphi\circ T\circ\varphi^{-1}$ so we have $\varphi\circ\Psi_{X-T(X)}(\lambda) = \log\tilde{T}(\exp(\varphi(\lambda Y)$ where  $Y=X-T(X)$.

Now using the technique of the classical case we have:

$\exp(\lambda\varphi(Y))$ is a convex function of $\varphi(Y)$, so that: \[\exp(\lambda\varphi(Y))\leq\frac{b-\hat{Y}}{b-a}\exp(\lambda a) + \frac{\hat{Y}-a}{b-a}\exp(\lambda b)\]
Hence \[\tilde{T}(\exp(\lambda\varphi(Y)))\leq \frac{b-\tilde{T}(\hat{Y})}{b-a}\exp(\lambda a) + \frac{\tilde{T}(\hat{Y})-a}{b-a}\exp(\lambda b)\]
let $h =\lambda (b-a)$ , $p=\frac{-a}{b-a}$ and $L(h)=-hp+\log(1-p+p\exp(h))$
Using the fact that $\tilde{T}(\hat{Y}) = 0$ to get \[ \frac{b-\tilde{T}(\hat{Y})}{b-a}\exp(\lambda a) + \frac{\tilde{T}(\hat{Y})-a}{b-a}\exp(\lambda b) = \exp(L(h))\]
Since $L(0) =L'(0)=0$ and $L''(h)\leq \frac{1}{4}$           for all $h$ 

by Taylor expansion we get, \[ L(h)\leq \frac{1}{8} \lambda^2 (b-a)^2\]

Hence \[\tilde{T}(\exp(\lambda\varphi(Y))\leq \exp(\frac{1}{8} \lambda^2 (b-a)^2)\]

Finally, \[T\circ\varphi^{-1}\circ\exp\circ\varphi(\lambda Y)\leq \varphi^{-1} \circ \exp(\frac{1}{8} \lambda^2 (b-a)^2)\varphi(u)\]

But, \[\varphi^{-1} \circ \exp(\frac{1}{8} \lambda^2 (b-a)^2)\varphi(u)= \varphi^{-1} \circ \exp\circ\varphi(\frac{1}{8} \lambda^2 (b-a)^2)u.\]  Which makes an end to our proof.

\end{proof}
\begin{corollary}
Let $E_u$ be a Dedekind complete Riesz space with a strong unit $u$ and $T$ a strictly positive conditional expectation on $E_u$ with $Tu=u$. Let  $X_1,..,X_n$ be $n$ T-independent element of $E_u$ such that $X_i$ in $[a_iu,b_iu]$ for all $1\leq i\leq n$, where $a_i,b_i$ are two different scalars. Then, for any strictly positive scalar $t$ we have:
\[TP_{(\sum_{i=1}^n (X_i-T(X_i)) - tu)^+} u \leq \exp(-\frac{2t^2}{\sum_{i=1}^n(b_i-a_i)^2})u\]
\end{corollary}

\bibliographystyle{plain}
\bibliography{biblio}

\end{document}